\documentclass[12pt,a4paper]{amsart}
\usepackage{}
\usepackage[all]{xy}
\usepackage{amsmath}
\usepackage{amsfonts}
\usepackage{amssymb}
\usepackage[all]{xy}           
\usepackage{mathrsfs}

\usepackage{bbding}
\usepackage{amscd}

\usepackage[shortlabels]{enumitem}

\usepackage{ifpdf}
\ifpdf
 \usepackage[colorlinks,final,backref=page,hyperindex]{hyperref}
\else
  \usepackage[colorlinks,final,backref=page,hyperindex,hypertex]{hyperref}
\fi

\usepackage{tikz}
\usepackage[active]{srcltx}
\usepackage[dvips]{lscape}
\xyoption{web}
\usetikzlibrary{arrows}
\usepackage{xspace,amsthm}
\usepackage[parfill]{parskip}
\usepackage{mathrsfs}
\makeatletter

\newtheorem{lemma}{Lemma}

\newtheorem{prop}[lemma]{Proposition}

\newtheorem{defi}[lemma]{Definition}
\newtheorem{theorem}[lemma]{Theorem}

\newcommand {\emptycomment}[1]{}

\def\z{\dot z}

\newcommand{\Z}{\mathbb{Z}}

\def\mh{\mathfrak{h}}

\newcommand{\e}{\epsilon}

\newcommand{\C}{\mathbb C}

\def\mg{\mathfrak{g}}

\def\md{\mathfrak{D}}
\def\sl{\mathfrak{sl}}
\def\gl{\mathfrak{gl}}
\def\sp{\mathfrak{sp}}
\def\ms{\mathfrak{s}}
\def\mn{\mathfrak{n}}

\begin{document}

\title[Harish-Chandra modules]{The category  of weight modules for    symplectic oscillator  Lie  algebras}

\author{Genqiang Liu}
\address{ School of Mathematics and Statistics, Henan University, Kaifeng 475004,
China}
\email{liugenqiang@amss.ac.cn}

\author{Kaiming Zhao}
\address{Department of Mathematics, Wilfrid
Laurier University, Waterloo, ON, Canada N2L 3C5,  and College of
Mathematics and Information Science, Hebei Normal University, Shijiazhuang 050016, Hebei, China}
\email{kzhao@wlu.ca}

\date{}

\begin{abstract}The rank $n$ symplectic oscillator  Lie  algebra $\mathfrak{g}_n$ is the semidirect product of the symplectic  Lie algebra $\mathfrak{sp}_{2n}$ and the Heisenberg Lie algebra $H_n$.  In this paper, we study weight modules with finite dimensional weight spaces over $\mathfrak{g}_n$. When $\dot  z\neq 0$, it is shown  that  there is an equivalence between the full subcategory $\mathcal{O}_{\mathfrak{g}_n}[\dot z]$  of the BGG  category $\mathcal{O}_{\mathfrak{g}_n}$ for $\mathfrak{g}_n$ and the BGG category   $\mathcal{O}_{\mathfrak{sp}_{2n}}$ for   $\mathfrak{sp}_{2n}$.  Then using the technique of localization and the structure of generalized highest weight modules,  we also give the  classification of  simple weight  modules over $\mathfrak{g}_n$ with finite-dimensional weight spaces.
\noindent
\end{abstract}



\maketitle

%

\section{Introduction}

Many important Lie algebras in  mathematical physics are finite dimensional but not semi-simple, such
as Schr{\"o}dinger algebras \cite{DDM},   conformal Galilei algebras \cite{GM, LMZ}, symplectic oscillator  Lie  algebras \cite{BS, AK}, Euclidean  algebras \cite{IK} and so on.
Unlike finite dimensional semi-simple Lie algebras, the representation theory of those Lie algebras is still not well developed \cite{Hu, M}. In this paper we will establish the representation theory for symplectic oscillator  Lie  algebras.

The rank $n$ symplectic oscillator  Lie  algebra $\mg_n$ is the the semidirect product of the symplectic  Lie algebra
$\sp_{2n}$ and the Heisenberg Lie algebra $H_n$.
This algebra is also called  Jacobi Lie algebra  in the literature, see \cite{BS,S}.  The universal enveloping
 algebra of $\mg_n$  is an infinitesimal Hecke algebra, see \cite{EGG}.
In mathematics, the Jacobi group is the semidirect product of the symplectic group and the Heisenberg group \cite{BS}.
The Jacobi group is an important object in connection with quantum
mechanics, geometric quantization, optics. The Jacobi groups were used to  describe the ``squeezed coherent states'' of quantum optics \cite{Be}. In   Number Theory,
automorphic forms on the Jacobi group are called Jacobi forms which has close relationship with the modular forms, see \cite{BS, EZ}.

 In this paper, we study the BGG category $\mathcal{O}_{\mg_n}$ for $\mg_n$
and give the classification of  simple Harish-Chandra modules for $\mg_n$.

 A classification of  simple Harish-Chandra modules over the following Lie algebras have been obtained,
 the Virasoro algebra \cite{M0},
 finite dimensional simple Lie algebras \cite{Fe, M},  conformal Galilei algebras \cite{LMZ}, the twisted Heisenberg Virasoro algebra \cite{LZ}, and affine Kac-Moody algebras  \cite{FT, DG} (the zero central charge case was claimed in \cite{DG}).

The paper is organized as follows.   In Section \ref{sec2} we provide the related definitions and notations.

In Section \ref{sec3} we first we establish an isomorphism from $U(\mg_n)/\langle z-\z\rangle$  to the associative algebra $U(\sp_{2n})\otimes \md_n$  for $\z\ne0$ (Proposition \ref{proposition3}), and show that any module in $\mathcal{O}_{\mg_n}[\z]$ with $\z\ne0$ is completely reducible over $H_n$ (Lemma \ref{eigen}). Then we prove that     the full subcategory $\mathcal{O}_{\mg_n}[\z]$ of $\mathcal{O}_{\mg_n} $ for $\mg_n$ with nonzero $\z$
  is equivalent to  the BGG category $\mathcal{O}_{\sp_{2n}}$ for   $\sp_{2n}$ (Theorem \ref{prop721}).

  The classification of all simple weight  modules with finite dimensional weight spaces for $\mg_n$
is obtained in Section \ref{sec4}. Theorem \ref{p14} gave all such $\mg_n$-modules with $z$ acts trivially,  which are actually simple parabolically induced $\sp_{2n}$-modules and  simple cuspidal $\sp_{2n}$-modules (See \cite{M}).   Theorem \ref{t15} gave all such $\mg_n$-modules with $z$ acts non-trivially, which consists of three classes: cuspidal $\mg_n$-modules in (a), parabolically induced   $\mg_n$-modules in (b), and a third class described in (c). We can see that the  third class in Theorem \ref{t15} (c) does not appear for finite-dimensional simple Lie algebras.
The representation theory of polynomial differential operator algebras, technique of localization and
 the structure of parabolically induced modules are widely used in our proofs.

Throughout  this paper, we denote by $\mathbb{Z}$, $\mathbb{Z}_+$, $\mathbb{N}$,
$\mathbb{C}$ and $\mathbb{C}^*$ the sets of  all integers, nonnegative integers,
positive integers, complex numbers, and nonzero complex numbers, respectively.
For any Lie algebra $\mg$, we denote its universal enveloping algebra by $U(\mg)$.

\section{Definitions and notations}
\label{sec2}
\subsection{The symplectic oscillator  Lie  algebra $\mg_n$}
We know that  the symplectic Lie algebra $\mathfrak{sp}_{2n}$  has  the natural representation on $\C^{2n}$ by left matrix  multiplication. Let $\{e_1, e_2, \cdots, e_{2n}\}$ be  the standard basis of $\C^{2n}$. The Heisenberg Lie algebra $H_{n}=\C^{2n}\oplus \C z$ is the Lie algebra with
Lie bracket   given by
$$[e_i, e_{n+i}]=z, \qquad [z, H_n]=0.$$
Recall that the symplectic oscillator  algebra $\mg_n$ is  the the semidirect product Lie algebra
$$\mg=\sp_{2n}\ltimes H_{n}.$$
Explicitly, the semidirect relations are $$[X,v]=Xv, \qquad [X, z]=0,$$ for all $X\in \mathfrak{sp}_{2n}, v\in \C^{2n}$.
When $n=1$, $\mg_1$ is the Schr{\"o}dinger algebra studied in \cite{DDM}. The representations of $\mg_1$ were studied
in \cite{D, DDM, DLMZ}.

The universal enveloping algebra $U(\mg_n)$ of the Jacobi Lie  algebra $\mg_n$ is   an infinitesimal Hecke algebra,
see  Example 4.11 in \cite{EGG}.

\subsection{Root space decomposition of $\mg_n$}

Recall that    $\mathfrak{sp}_{2n}$ is the Lie subalgebra of $\mathfrak{gl}_{2n}$ consisting of all $2n\times 2n$-matrices $X$ satisfying $SX=-X^{T}S$ where
\begin{displaymath}
 S=\left( \begin{array}{cc}
  0 & I_{n}\\
  -I_{n} & 0\\
 \end{array} \right).
\end{displaymath}
Equivalently, $\mathfrak{sp}_{2n}$ consists of all $2n\times 2n$-matrices with block form
\begin{displaymath}
 \left( \begin{array}{cc}
  A & B\\
  C & -A^T\\
 \end{array} \right)
\end{displaymath}
such that $B= B^{T}$,  $C= C^{T}$. Let $e_{ij}$ denote the matrix unit whose $(i,j)$-entry is $1$ and $0$ elsewhere.
Then $\mh=\text{span}\{h_{i}:=e_{i,i} - e_{n+i,n+i} \mid 1 \leq i \leq n\}$  is the standard Cartan subalgebra of $\mathfrak{sp}_{2n}$
and
$$\mathfrak{h}_n=\C h_1\oplus \cdots \oplus \C h_n \oplus \C z$$ is a Cartan subalgebra
of $\mg_n$.
Let $\{\epsilon_{i}\}\subset\mathfrak{h}_n^{*}$ be such that $\epsilon_{i}(h_{k})=\delta_{i,k}$ and $\epsilon_{i}(z)=0$.
The root system of $\mg_n$ is precisely \[\Delta= \{\pm \epsilon_{i} \pm \epsilon_{j},  \pm\epsilon_{j} | 1 \leq i,j \leq n\} \setminus \{0\}
=\Delta_{\sp_{2n}}\cup  \{\pm\epsilon_{j} | 1 \leq j \leq n\} \]
where $\Delta_{\sp_{2n}}$ is the root system of ${\sp_{2n}}$.
The positive root system is
$$\Delta_+=\{\epsilon_{i}-\epsilon_{j}, \epsilon_{k}+\epsilon_{l}, \epsilon_{k}, \mid 1\leq i<j\leq n, 1\leq k, l\leq n\}.$$

We list root vectors in $\mg_n$ as follows. The indices $i,j$ are integers between $1$ and $n$, with $i\ne j$ when we encounter $\epsilon_{i}-\epsilon_{j}$.
\begin{displaymath}
\begin{array}{r c l | c}
\multicolumn{3}{c}{\text{Root vector}} &\text{Root} \\
\hline
X_{\epsilon_{i}}              &:=& e_{i}                 & \epsilon_{i}\\
X_{-\epsilon_{i}}             &:=& e_{n+i}                 & -\epsilon_{i}\\
X_{\epsilon_{i}+\epsilon_{j}}  &:=& e_{i,n+j} + e_{j,n+i}      & \epsilon_{i}+\epsilon_{j}\\
X_{-\epsilon_{i}-\epsilon_{j}} &:=& e_{n+i,j} +e_{n+j,i}      & -\epsilon_{i}-\epsilon_{j}\\
X_{\epsilon_{i}-\epsilon_{j}}  &:=& e_{i,j} - e_{n+j,n+i}      & \epsilon_{i}-\epsilon_{j}\\
\end{array}
\end{displaymath}
Then we obtain a basis of $\mg_n$ as follows $$B:=\{X_{\alpha} | \alpha \in \Delta\} \cup \{h_{i} , z| 1 \leq i \leq n\}.$$

Set
\begin{displaymath}
\mathfrak{n}_{\pm}:=\bigoplus_{\alpha\in \Delta_{\pm}}\mathfrak{\mg}_\alpha.
\end{displaymath}
Then the decomposition
\begin{equation}\label{eq1}
\mg_n=\mathfrak{n}_{-}\oplus\mathfrak{h}_n\oplus \mathfrak{n}_{+}
\end{equation}
is a {\em triangular decomposition} of $\mathfrak{\mg}_n$ which  tells us that
\begin{displaymath}
U(\mg_n)\cong U(\mathfrak{n}_{-})\otimes U(\mathfrak{h}_n)\otimes U(\mathfrak{n}_{+}).
\end{displaymath}
The Lie subalgebra $\mathfrak{b}:=\mathfrak{h}_n\oplus \mathfrak{n}_{+}$  is   a Borel subalgebra of $\mg_n$.

\subsection{Weight modules}
 A $\mg_n$-module  $M$ is called a {\it weight module} if $\mathfrak{h}_n$ acts diagonally on  $M$, i.e.
$$ M=\oplus_{\lambda\in \mathfrak{h}_n^*} M_\lambda,$$
where $M_\lambda:=\{v\in M \mid hv=\lambda v\}.$   Denote $$\mathrm{Supp}(M):=\{\lambda\in \mathfrak{h}_n^* \mid M_\lambda\neq0\}.$$

A weight module is called a {\it Harish-Chandra module} if all its weight spaces are finite dimensional.
For a weight module $M$, a
weight vector $v\in M_\lambda$ is called a {\it highest weight vector} if $\mathfrak{n}_{+}v=0$.
 A module is called a {\it highest weight module}
if it is generated by a highest weight vector.  A module is called to  be  {\it uniformly bounded}
if the dimensions of all weight spaces are bounded by a fixed integer.

\subsection{Verma modules}

For $\lambda\in\mathfrak{h}^*$ and $\z\in\C$ denote by $\mathbb{C}_{\lambda, \z}$ the one-dimensional $\mathfrak{b}$-module
with the generator $v_{\lambda}$ and the action given by
\begin{displaymath}
\mathfrak{n}_+ \mathbb{C}_{\lambda, \z}=0,\qquad z \cdot v_{\lambda}=\dot z v_{\lambda}, \qquad h \cdot v_{\lambda}=\lambda(h) v_{\lambda}\text{ for all } h\in\mathfrak{h}.
\end{displaymath}
The {\em Verma module} is defined, as usual, as follows:
\begin{displaymath}
M(\z, \lambda):=\mathrm{Ind}_{\mathfrak{b}}^{\mg_n}\mathbb{C}_{\lambda, \z}\cong
U(\mg_n)\bigotimes_{U(\mathfrak{b})}\mathbb{C}_{\lambda, \z}.
\end{displaymath}
For convenience,  we also denote  $1\otimes v_{\lambda}$ of $M(\z,\lambda)$
simply by $v_{\lambda}$.
Let  $K(\z,\lambda)$  be the unique maximal proper submodule of $M(\z,\lambda)$.
The quotient module $L(\z,\lambda)=M(\z,\lambda)/K(\z,\lambda)$ is the
 unique simple quotient module of $M(\z,\lambda)$. Similarly, we have the modules
 $M_{\sp_{2n}}(\lambda)$ and $L_{\sp_{2n}}(\lambda)$ for $\sp_{2n}$.

\section{BGG category}\label{sec3}
In this section, we study the  BGG category $\mathcal{O}_{\mg_n} $ for $\mg_n$.

\subsection{The BGG category $\mathcal{O}_{\mg_n} $}\label{Odef}

\begin{defi} The BGG category $\mathcal{O}_{\mg_n} $ for $\mg_n$ is the full subcategory of $\emph{Mod}\,U(\mg_n)$ (the category of all left $U(\mg_n)$-modules) whose objects $M$ are
the modules satisfying the following three conditions.
\begin{enumerate}[$($a$)$]
\item $M$ is a finitely generated $U(\mg_n)$-module.
\item $M$ is a weight module.
\item $M$ is locally $\mn_+$-finite, i.e., for each $v\in M$, the subspace $U(\mn_+)v$ is finite dimensional.
\end{enumerate}
\end{defi}

Similarly, we have  the BGG category $\mathcal{O}_{\sp_{2n}}$ for $\sp_{2n}$.

For information  on the BGG category $\mathcal{O} $ for semisimple Lie algebras, one can see   \cite{BGG, Hu}.
By the standard arguments, we can prove that the category $\mathcal{O}_{\mg_n} $  has the following  property.

\begin{lemma}\label{basic-property}Let $M$ be any  object in $\mathcal{O}_{\mg_n} $.
\begin{enumerate}[$($1$)$]
\item The module $M$ has a finite filtration
$$0 = M_0 \subset  M_1\subset \cdots  \subset M_m = M$$
with each factor $M_j/M_{j-1}$ for $1 \leq j\leq m$ being a highest weight module.
\item  Each weight space of $M$  is finite dimensional.
\item Any simple object in  $\mathcal{O}_{\mg_n} $ is isomorphic to some $L(\z, \lambda)$
for $\lambda \in\mh^*, \z\in \C$.
\end{enumerate}
\end{lemma}
So  highest weight modules are the basic elements of   the category $\mathcal{O}_{\mg_n} $.

 Next we will introduce the Shale-Weil representation for $\mg_n$, which is important
 for our later arguments on the category $\mathcal{O}_{\mg_n} $.
\subsection{Shale-Weil representation}

 For $n\in\mathbb{N}$,  denote by $\md_{n}$ the algebra of polynomial differential operators of $\C[t_1,\cdots,t_n]$, called the {\it Weyl algebra of rank} $n$. Namely,
$\md_{n}$  is the associative algebra  over $\C$ generated by $2n$-indeterminates $t_1,\dots,t_n$,
$\partial_1,\dots,\partial_n$ subject to
the relations
$$[\partial_i, \partial_j]=[t_i,t_j]=0,\qquad [\partial_i,t_j]=\delta_{i,j},\ 1\leq i,j\leq n.$$
A $\md_n$-module $V$ is a {\it weight module} if all $t_i\partial_i$ are semisimple on $V$.
For a weight module $V$, a
weight vector $v\in V_\lambda$ is called a {\it highest weight vector} if $\C[t_1, t_2, \cdots, t_n]v=0$.
 A $\md_n$-module $V$ is called a {\it highest weight module} if it is generated by a highest weight vector.

 \begin{prop}\label{proposition3}For any nonzero scalar $\z$, we have the     associative algebra isomorphism $ \phi_{\z}: \,\, U(\mg_n)/\langle z-\z\rangle\rightarrow U(\sp(2n))\otimes \md_n$ defined by
$$\aligned
          & X_{\e_i-\e_j}\mapsto  X_{\e_i-\e_j}\otimes 1+ 1\otimes t_i\partial_j, (i\ne j),\\
          & X_{\e_i+\e_j}\mapsto  X_{\e_i+\e_j}\otimes 1+ 1\otimes t_it_j,\\
          & X_{-\e_i-\e_j}\mapsto  X_{-\e_i-\e_j}\otimes 1- 1\otimes \partial_i \partial_j,\\
          & h_i\mapsto h_i\otimes 1 +1\otimes t_i\partial_i+\frac{1}{2},\\
         & X_{\epsilon_{i}}\mapsto 1\otimes \sqrt{\z}t_i,\\
          &X_{-\epsilon_{i}}\mapsto -1\otimes \sqrt{\z}\partial_i,\\
          \endaligned
$$ where $1\leq i, j \leq n$.
 \end{prop}
 \begin{proof}  We can directly verify that the linear map $f_{\z}: \,\,  \mg_{n} \rightarrow  (\md_n, [, ])$ defined by
  \begin{equation}\label{weil-rep}\aligned
          & X_{\e_i-\e_j}\mapsto  t_i\partial_j, (i\ne j),\\
          & X_{\e_i+\e_j}\mapsto  t_it_j,\\
          & X_{-\e_i-\e_j}\mapsto   -\partial_i \partial_j,\\
          & h_i\mapsto t_i \partial_i+\frac{1}{2},  \\
          & X_{\epsilon_{i}}\mapsto\sqrt{\z} t_i,\\
          &X_{-\epsilon_{i}}\mapsto -\sqrt{\z}\partial_i,\\
          \endaligned
          \end{equation}
for $1\leq i, j \leq n$, is a Lie algebra homomorphism. (Note that the above embedding for $\sp_{2n}$ is similar to \cite{BL}). Using PBW Theorem we   have the  surjective   associative algebra homomorphism $ \theta_{\z}: \,\, U(\mg_{n})\rightarrow  \md_n$ defined by extending $f_{\z}$.

It is clear that $U(\mg_n)/\langle z-\z\rangle=U(\sp_{2n})\otimes U(H_n)/\langle z-\z\rangle$ where the right-hand side is just considered as a vector space tensor product, not tensor product of the two  associative algebras.

We have the  algebra homomorphism
 $ \iota_1: \,\, U(\sp_{2n})\rightarrow U(\sp_{2n})\otimes \md_n$ defined by
$$\aligned
          & X_{\e_i-\e_j}\mapsto  X_{\e_i-\e_j}\otimes 1+ 1\otimes t_i\partial_j, (i\ne j),\\
          & X_{\e_i+\e_j}\mapsto  X_{\e_i+\e_j}\otimes 1+ 1\otimes t_it_j,\\
          & X_{-\e_i-\e_j}\mapsto  X_{-\e_i-\e_j}\otimes 1- 1\otimes \partial_i \partial_j,\\
          & h_i\mapsto h_i\otimes 1 +1\otimes t_i\partial_i+\frac{1}{2},\\
          \endaligned
$$ where $1\leq i, j \leq n$. Note that $\iota_1$ is injective.
We also  have the algebra isomorphism  $ \iota_2: \,\, U(H_n)/\langle z-\z\rangle\rightarrow  \md_n$ defined by
$$ X_{\epsilon_{i}}\mapsto\sqrt{\z} t_i,\,\,\,
          X_{-\epsilon_{i}}\mapsto -\sqrt{\z}\partial_i, \forall 1\leq i \leq n.$$
We can eaily see that  $ \phi_{\z}$ is a surjective algebra homomorphism, and $$U(\sp(2n))\otimes \md_n=\iota_1(U(\sp_{2n}))\otimes \iota_2(U(H_n)/\langle z-\z\rangle),$$
where the right-hand side is just considered as a vector space tensor product.
Then the linear map (not considered as algebra homomorphism) $$\phi_{\z}=\iota_1\otimes\iota_2:U(\sp_{2n})\otimes U(H_n)/\langle z-\z\rangle \to \iota_1(U(\sp_{2n}))\otimes \iota_2(U(H_n)/\langle z-\z\rangle)$$
is a vector space isomorphism.
Therefore $\phi_{\z}$ is an algebra   isomorphism. The proof is complete.
 \end{proof}

Note that the restriction of $\phi_{\z}:U(H_n)/\langle z-\z\rangle\to \md_n$ is an algebra isomorphism for $\z\ne0$.
Later in this paper will identify these two associative algebras.

For any $U(\sp_{2n})$-module $V$, $\md_{n}$-module $N$, via the homomorphism $\phi_{\z}$, $V\otimes N$ can be viewed as a  $\mg_n$-module which is denoted by
${V}\otimes _{\z}{N}$.  In particular, the simple $\md_{n}$-module $$S=(\C[t_1^{\pm 1}]/\C[t_1] )\otimes (\C[t_2^{\pm 1}]/\C[t_2] )\otimes\cdots \otimes(\C[t_n^{\pm 1}]/\C[t_n])$$becomes a $\mg_n$-module through $\theta_{\z}$, denoted by $S_{\z}$, which is isomorphic to the simple highest weight $\mg_n$-module
$L(\z, -\frac{1}{2}(\epsilon_1+\cdots+\epsilon_n))$.  This module is called the {\it Shale-Weil module} over $\mg_n$  (see \cite{M}).

\subsection{The structure of highest weight modules  over $\mg_n$}

\begin{prop}\label{thm-highest}
Let $ \z\in  \C, \lambda\in\mh^*$.
\begin{enumerate}[$($a$)$]
\item If $\z=0$, then $H_n L(\z, \lambda)=0$.
\item
If $\z\neq 0$, then we have that
$$M(\z, \lambda)\cong  {M}_{\sp_{2n}}(\lambda')
\otimes_{\z}S , $$
$$L(\z, \lambda)\cong  {L}_{\sp_{2n}}(\lambda')
\otimes_{\z}S , $$
where $\lambda'=\lambda
+\frac{1}{2}(\epsilon_1+\cdots+\epsilon_n)$.
Consequently,
$M(\z, \lambda)$ is simple if and only if $M_{\sp_{2n}}(\lambda+\frac{1}{2}(\epsilon_1+\cdots+\epsilon_n))$ is simple.
\end{enumerate}
\end{prop}

\begin{proof}(a) In $M(\z, \lambda)$ we have  \begin{equation}\label{hw}X_{\e_j} X_{-\e_k}v_\lambda=X_{\e_i+\e_j} X_{-\e_k}v_\lambda=0, \end{equation} and
\begin{equation*}X_{\e_i-\e_j} X_{-\e_k}v_\lambda=-\delta_{ik} X_{-\e_j}v_\lambda,\end{equation*} yielding that
$$\sum_{k=1}^n U(\mg_n)X_{-\e_k}v_\lambda= \sum_{k=1}^n U(\mn_ -)X_{-\e_k}v_\lambda,$$ which is a proper submodule of
$M(\z, \lambda)$. So $X_{-\e_k}v_\lambda\in K(\z, \lambda)$, i.e,  $X_{-\e_k}v_\lambda=0$ in $L(\z, \lambda)$, for any $k\in\{1,\dots,n\}$.
By the fact that $H_n$ is an ideal of $\mg_n$  and the simplicity of $L(\z, \lambda)$,  we have $H_n L(\z, \lambda)=0$.

(b)  Let $v_{\lambda'}$ be a  highest weight vector of $M_{\sp_{2n}}(\lambda')$. It is not hard to  see that ${M}_{\sp_{2n}}(\lambda')
\otimes_{\z}S_{\z}$ is a highest weight $\mg_n$-module with highest weight $\lambda'-\frac{1}{2}(\epsilon_1+\cdots+\epsilon_n)$.
Next we prove that $${M}_{\sp_{2n}}(\lambda')
\otimes_{\z}S_{\z}\simeq M(\z, \lambda'-\frac{1}{2}(\epsilon_1+\cdots+\epsilon_n)).$$

Take  $X\in U(\mn_-)$ with
\begin{equation}\label{X} X  (v_{\lambda'}\otimes_{\z} t_1^{-1} t_2^{-1}\cdots  t_n^{-1})=0. \end{equation}
Using a PBW basis, we can write
$$X=\sum _{\alpha\in \Z_+^n}f_{\alpha}(X_{\e_i-\e_j}, X_{-\e_i-\e_j})\prod_{1\le i\le n}X_{-\e_i}^{\alpha_i}.$$
where $f_{\alpha}(X_{\e_i-\e_j}, X_{-\e_i-\e_j})\in U(\mn_-)\cap U(\sp_{2n})$.
 From (\ref{X}) and Proposition 3, by induction on    $|\alpha|$ we  deduce that
 $f_{\alpha}(X_{\e_i-\e_j}, X_{-\e_i-\e_j})=0$ for all $\alpha$. Thus $X=0$, yielding that
 ${M}_{\sp_{2n}}(\lambda')
\otimes_{\z}S \simeq M(\z, \lambda'-\frac{1}{2}(\epsilon_1+\cdots+\epsilon_n))$.

 Since ${L}_{\sp_{2n}}(\lambda')
\otimes S$ is a simple $U(\sp_{2n})\otimes \md_n$-module, from the definition of  ${L}_{\sp_{2n}}(\lambda')\otimes_{\z}S $  and Proposition 3 we
  see that it is a   simple $\mg_n$-module. Thus the second isomorphism in (b) follows.
\end{proof}

\subsection{The characterizations of $\mathcal{O}_{\mg_n}[\z]$ }

We denote by $\mathcal{O}_{\mg_n}[\z]$ the full subcategory of
$\mathcal{O}_{\mg_n} $ consisting of all modules $M\in \mathcal{O}_{\mg_n}$ such that
$z$ acts on $M$ as a scalar $\z$.

\begin{lemma}\label{eigen}Let $ M\in \mathcal{O}_{\mg_n}[\z]$ with   $\dot{z}\ne0$. Let $v$ be a nonzero
weight vector in $M$. Then $U(H_n)v=V_1\oplus \cdots \oplus V_k$, with $V_i\cong S_{\z}$ for each $i\in\{1,\dots, k\}$.
\end{lemma}

\begin{proof} Since $U(H_n)/\langle z-\z\rangle \cong \md_n $, we can consider that $U(H_n)v=\md_n v$ which is still a weight module with respect to $\mh$. Since $M$ has finite dimensional weight spaces, we see that
 $$\dim \C[t_1\partial_1,\cdots,t_n\partial_n]v< \infty.$$
 Note that as vectors spaces
 $$\md_n=\big(\C[t_1,\cdots,t_n] +\C[\partial_1,\cdots,\partial_n] \big)\otimes  \C[t_1\partial_1,\cdots,t_n\partial_n].$$
 Thus the $\md_n$-module $\md_n v$  has finite composition length $k$. We will prove the statement by induction on $k$.

 If $k=1$, then $\md_n v$ is a simple highest weight
  module, i.e., $\md_n v\cong S_{\z}$. Let $V$ be a maximal $\md_n$-submodule of $\md_n v$. By the
  induction hypothesis, $V=V_1\oplus \cdots \oplus V_{k-1}$, with $V_i\cong S_{\z}$ for each $1 \leq i\leq k-1$,  and there is an
  epimorphism $\gamma:\md_n v\rightarrow S_{\z}$ such that $\ker(\gamma)=V$. Since $ \md_n v/V$ is a highest weight module over $\md_n$, we take a highest weight vector  $\bar w\in \md_n v/V$ with $w\in  \md_n $ being a weight vector. Then   $t_iw\in V$  for any $1\leq i\leq n$.

\noindent{\bf Claim:} There exists $v'\in V$ such that $t_i(w+v')=0$ for any $1\leq i\leq n$.

If $t_1w\neq 0$, we choose the smallest integer $l$ such that $t_1^{l+1}w=0$.
From $$t_1\partial_1^it_1^i=\partial_1^it_1^{i+1}-i\partial_1^{i-1}t_1^i,$$ we have that
$$\aligned t_1 (w+ \sum_{i=1}^l \frac{1}{i!}\partial_1^it_1^iw)&=t_1w+ \sum_{i=1}^l \frac{1}{i!}\partial_1^it_1^{i+1}w
-\sum_{i=1}^l \frac{1}{(i-1)!}\partial_1^{i-1}t_1^iw\\
&=\sum_{i=1}^{l-1} \frac{1}{i!}\partial_1^it_1^{i+1}w
-\sum_{i=2}^l \frac{1}{(i-1)!}\partial_1^{i-1}t_1^iw\\
&=0.\endaligned $$
Note that $\sum_{i=1}^l \frac{1}{i!}\partial_1^it_1^iw\in V$. By repeating the above arguments, we can have
$v'\in V$ such that $t_i(w+v')=0$ for any $1\leq i\leq n$. 
The claim follows. 

Then  $\md_n(w+v')\cong S_{\z}$.
By the simplicity of $\md_n(w+v')$, we see that $\md_n v=V\oplus \md_n(w+v')\cong V\oplus S_{\z}$. 
 Then the proof is complete.
\end{proof}

\begin{prop}\label{mainlemma} Let $\dot{z}\ne0$. For any $ M\in \mathcal{O}_{\mg_n}[\z]$,
there is an $N\in \mathcal{O}_{\sp_{2n}}$   such that
$M \cong  {N}\otimes_{\z}  S$.
\end{prop}

\begin{proof}
Consider $M$ as a $U(\sp_{2n})\otimes \md_n$-module via the isomorphism $\phi_{\z}$ in Proposition 3 where actually $\phi_{\z}(U(H_n)/\langle z-\z\rangle)=\md_n$. By Lemma \ref{eigen}, we know that $M$ is a direct sum of irreducible highest weight $\md_n$-modules  all isomorphic.

Let $V$ be the subspace of $M$ consisting all highest weight vectors over $\md_n$. We see that $U(\sp_{2n})V=V$. On the other hand we have $M=V\otimes S$. Thus going back to the $\mg_n$-module
$M$ we see that $M=V\otimes_{\z}S$.\end{proof}

\begin{theorem}\label{prop721} If $\z\neq 0$,
then the functor
\begin{displaymath}
-\otimes_{\z}   S: \mathcal{O}_{\sp_{2n}}\to \mathcal{O}_{\mg_n}[\dot z],
\end{displaymath}
 is an equivalence of the two categories.
\end{theorem}

\begin{proof}Denote $\mathcal{F}=-\otimes _{\z} S$. Note that the automorphisms of simple $\md_n$-modue $S$ is $\C^*$. From Proposition 5, we see that
the homomorphism $$\mathcal{F}_{V,W}: \text{Hom}_{\sp_{2n}}(V,W)\rightarrow\text{Hom}_{\mg_n}(\mathcal{F}(V),\mathcal{F}(W))$$ is isomorphism,
 for any modules $V,W\in \mathcal{O}_{\sp_{2n}}$. Again
from Proposition \ref{mainlemma},  $\mathcal{F}$ is an equivalence of the two categories.
\end{proof}

\noindent{\bf Remark:}  When $\z=0$, the Verma module $M(0, \lambda)$ has infinite composition length. Then subcategory
$\mathcal{O}_{\mg_n}[0]$ has more complicated structure. For $n=1$, the discussion of the category
$\mathcal{O}_{\mg_n}[0]$  can be referred to \cite{DLMZ}.  Some properties  of  the category $\mathcal{O} $ over
 a deformation of the symplectic oscillator algebra were given in \cite{GK}.

\section{Classification of simple Harish-Chandra $\mg_n$-modules}\label{sec4}

Before giving the classification of simple Harish-Chandra modules  over  $\mg_n$, we will first study  generalized highest weight modules.
\subsection{Generalized highest weight modules}\label{GHWD}
Denote
$$\aligned \mg_n^+=& (\oplus_{ 1\leq i, j\leq n} \C X_{\e_i+\e_j})\oplus (\oplus_{1\leq i \leq n} X_{\e_i}), \\
\mg_n^-=&(\oplus_{ 1\leq i, j\leq n} \C X_{-\e_i-\e_j})\oplus (\oplus_{1\leq i \leq n} X_{-\e_i}),\\
\mg_n^0= & \mh_n\oplus (\oplus_{ 1\leq i\neq j\leq n} \C X_{\e_i-\e_j}).\endaligned$$

Then we have a   different  triangular decomposition
$$ \mg_n=\mg_n^-\oplus \mg_n^0\oplus \mg_n^+.$$

Let $V$ be a simple  $\mg_n^0\oplus \mg_n^+$-module with $\mg_n^+V=0$ and $zv=\z v$ for any
$v\in V$, where $\z\in\C$.  Note that $\mg_n^0=\gl_n\oplus \C z$.  So  $V$ is actually a simple $\gl_n$-module.
The induced $\mg_n$-module
$$M(\z, V)=U(\mg_n)\otimes_{U(\mg_n^0\oplus \mg_n^+)}V$$ is called a {\it generalized Verma module} over $\mg_n$.
We denote its unique
irreducible quotient module by  $L(\z, V)$. By the analogous construction, we have the $\sp_{2n}$-modules
 $M_{\sp_{2n}}(V) $ and $ L_{\sp_{2n}}(V)$, for any simple $\gl_n$-module $V$.

Similar to Proposition \ref{thm-highest}, we have the next result.

\begin{prop}\label{GHM} \begin{enumerate}[$($a$)$]
\item If $\z=0$, then $H_nL(\z, V)=0$.
\item If $\z\neq 0$, then we have that  $$M(\z, V)\cong  {M}_{\sp_{2n}}(V)\otimes_{\z}S, $$
and $$L(\z, V)\cong  {L}_{\sp_{2n}}(V)\otimes  _{\z}S.$$
\end{enumerate}
\end{prop}

\begin{proof}(a) The proof is similar to that for Proposition 4 (a). We omit the details.

(b)    It is not hard to  see that ${M}_{\sp_{2n}}(V)
\otimes_{\z}S$ is a generalized highest weight $\mg_n$-module.
Next we prove that $$M(\z, V)\cong  {M}_{\sp_{2n}}(V)\otimes_{\z}S.$$

It is enough to show that, for any linearly independent $v_k\in V$ ($k=1,2,\cdots,r$) and   $X_k\in U(\mg_n^-)$ with
\begin{equation}\label{X_k} \sum_{k=1}^rX_kv_k =0,\end{equation}
we can deduce that $X_k=0$ for all $k$.
Using a PBW basis, we can write
$$X_k=\sum _{\alpha\in \Z_+^n}f^{(k)}_{\alpha}( X_{-\e_i-\e_j})\prod_{1\le i\le n}X_{-\e_i}^{\alpha_i}.$$
where $f^{(k)}_{\alpha}( X_{-\e_i-\e_j})\in U(\mg_n^-)\cap U(\sp_{2n})$.
 From (\ref{X_k}) and Proposition 3, by induction on    $|\alpha|$ we  deduce that
 $f^{(k)}_{\alpha}(  X_{-\e_i-\e_j})=0$ for all $\alpha$. Thus all $X_k=0$, yielding that
   ${M}_{\sp_{2n}}(V)\otimes_{\z}S$ is a generalized Verma module over $\mg_n$, i.e.,
   $M(\z, V)\cong  {M}_{\sp_{2n}}(V)\otimes_{\z}S.$

 Since ${L}_{\sp_{2n}}(V)\otimes_{\z}S$ is a simple $U(\sp_{2n})\otimes \md_n$-module, from the definition of  ${L}_{\sp_{2n}}(V)\otimes_{\z}S$  and Proposition \ref{proposition3} we
  see that it is a   simple $\mg_n$-module. Thus the second isomorphism in (b) follows.
\end{proof}

By Lemma 3.3 in \cite {DMP}, we have the following Lemma.

\begin{lemma} \label{LN}Let $M$ be a weight $\mg_n$-module with all finite dimensional weight spaces, $\alpha\in \Delta_{\sp_{2n}}$.
The following two conditions are equivalent.
\begin{enumerate}[$($a$)$]
\item  For each $\lambda\in \text{Supp}M$, the space $M_{\lambda+k\alpha}$ is zero for all but finitely many $k>0$.
\item The element $X_\alpha$ acts locally nilpotently on $M$.
\end{enumerate}
\end{lemma}

\begin{lemma} \cite{DMP}\label{Nil} If $M$ is a simple weight $\mg_n$-module,
then the action of $X_{\alpha}$ on $M$ is injective or locally  nilpotent, for any $\alpha\in \Delta$.
\end{lemma}

Next, we will introduce a lemma which will be used in studying the structure of generalized highest modules.

\begin{lemma}  \label {lemma-nil}Let $M$ be any simple Harish-Chandra
$\mg_n$-module, on which  $z$ acts as some $\z\in\C$.
\begin{enumerate}[$($a$)$]
\item  If $X_{2\e_i}$ acts locally nilpotently on $M$, then so does $X_{\e_i}$.
If $\z\neq 0$,  then  the converse is also true.
\item  If $X_{2\e_i}$  and $X_{2\e_j}$ acts  locally nilpotently on $M$ with $i\neq j$, then $X_{\e_i+\e_j}$ also acts  locally nilpotently on $M$.
\end{enumerate}
 \end{lemma}
 \begin{proof}(a)  If $X_{2\e_i}$ acts locally nilpotently on $M$, then by Lemma \ref{LN}, for each $\lambda\in \text{Supp}M$, the space $M_{\lambda+2k\e_i}$ is zero for all but finitely many $k>0$. Thus   $X_{\e_i}$ acts  locally nilpotently on $M$.

 Conversely,  if $\z\neq 0$, $X_{\e_i}$ acts nilpotently on $M$ but  $X_{2\e_i}$ acts injectively on $M$, then for each $\lambda\in \text{Supp}M$, the space $M_{\lambda+2k\e_i} \neq 0$ for any $k\in\mathbb{Z}_+$. By the nilpotence of $X_{\e_i}$,
we have an infinite sequence of positive integers $k_1<k_2<\cdots< k_l<\cdots$ and nonzero
$v_{k_l}\in M_{\lambda+k_l\e_i} $ such that $X_{\e_i} v_{k_l}=0$  for each $ l \in\mathbb{N}$. Using that $\z\neq 0$, we can see that
$ \{X_{-\e_i}^{k_l}v_{k_l}\mid l \in\mathbb{N}\} $ is a linearly independent infinite subset of $M_\lambda$, contradicting
 the fact that  $\dim M_\lambda<\infty$.

 (b) Suppose that  $X_{\e_i+\e_j}$  acts injectively on $M$.

 We claim that  both $X_{-2\e_i}$ and $X_{\e_j-\e_i}$ acts nilpotently on $M$.
  If $X_{\e_j-\e_i}$  acts injectively on $M$, then from $[X_{\e_j-\e_i}, X_{\e_i+\e_j}]=X_{2\e_j}$ and Lemma \ref{LN}, we
 see that $X_{2\e_j}$ acts injectively on $M$, contradicting the nilpotence of  $X_{2\e_j}$.
 If $X_{-2\e_i}$ acts injectively on $M$, then from $[X_{\e_i+\e_j}, X_{-2\e_i}]=2X_{\e_j-\e_i}$ and Lemma \ref{LN}, we
 see that $X_{\e_j-\e_i}$ acts injectively on $M$, contradicting the nilpotence of  $X_{\e_j-\e_i}$.

 Let $\ms_{i}$ be the subalgebra generated by $X_{2\e_i}$, $X_{-2\e_i}$ which is isomorphic to $\sl_2$.
 The   nilpotence of $X_{2\e_i}$ and  $X_{-2\e_i}$ tells us that $M$ is a sum of finite dimensional irreducible
 $\ms_{i}$-modules. So the weight set of $M$ is invariant under the action of the reflection $r_{2\e_i}$.
  Since $X_{\e_i+\e_j}$  acts injectively on $M$, for each $\lambda\in \text{Supp}M$,
  the space $M_{\lambda+k\e_i+k\e_j} \neq 0$ for any $k\in\mathbb{N}$. Then
  $$r_{2\e_i}(M_{\lambda+k\e_i+k\e_j})= M_{\lambda-\langle\lambda,\e_i\rangle\e_i+k\e_j-k\e_i} \neq 0,$$ for any $k\in\mathbb{N}$.
 This implies that $X_{\e_j-\e_i}$ acts injectively on $M$,  a contradiction. So (b) follows.
 \end{proof}

\begin{prop}\label{CGHM}Let $M$ be a simple  weight $\mg_n$-module with all finite dimensional weight spaces.
If $X_{2\epsilon_i}$ acts  locally nilpotently on $M$ for any $1\leq i\leq n$, then
$M\cong L(\z, V)$ for some simple Harish-Chandra $\gl_n$-module $V$.
\end{prop}

\begin{proof}  By Lemma \ref{lemma-nil}, $X_{\epsilon_i}, X_{\epsilon_i+ \epsilon_j} $ act  locally nilpotently on $M$ for any $1\leq i\leq n$.
Then we can find a weight vector $v$ which is annihilated by $\mg_n^+$. Let $V=\{v\in M\mid \mg_n^+v=0\}$ which has to be  a simple
$\mg_n^0$-module.
Consequently, $M\cong L(\z, V)$. Since $M$ is a Harish-Chandra $\mg_n$-module, $V$ has to be a Harish-Chandra $\gl_n$-module.
\end{proof}

\subsection{Simple Harish-Chandra modules}
In this subsection, we will give the classification of simple weight $\mg_n$-modules with finite dimensional weight spaces. Our method is quite different from those in \cite{M, DMP} where they broke the whole system into different subsets according to  actions' nilpotency of  root vectors.
We need only to use all the root vectors with respect to  long roots $\pm 2\epsilon_i$ as you will see next.

For any simple Harish-Chandra
$\mg_n$-module $M$,  we set $$\aligned I_M=& \{ 1\leq i \leq n\mid  X_{2\epsilon_i} |_M\   \text{and} \ X_{-2\epsilon_i}|_M \ \text{are injective}\},\\
F_M=& \{1\leq i \leq n\mid  X_{2\epsilon_i} |_M\   \text{and} \ X_{-2\epsilon_i}|_M \ \text{locally nilpotent}\},\\
F^+_M=& \{ 1\leq i \leq n\mid  X_{2\epsilon_i}|_M\   \text{is locally nilpotent}, \text{but}\   \ X_{-2\epsilon_i}|_M \ \text{is injective} \},\\
F^-_M=& \{ 1\leq i \leq n\mid  X_{2\epsilon_i}|_M\   \text{is injective}, \text{but}\   \ X_{-2\epsilon_i}|_M \ \text{is locally nilpotent} \}.
\endaligned$$
Then $I_M \cup F_M \cup F^+_M\cup F^-_M=\{1,\dots,n\}$.

Consider the multiplicative subset $$S(I)= \{ \prod_{i\in I_M}X_{-2\epsilon_i}^{r_i} \mid r_i\in\Z_+\}$$ which
  is an Ore subset of $U(\mg_n)$,  and hence we have the corresponding Ore
localization $U(\mg_n)_{S(I)}$,  see \cite{B, M}.  Assume that $I_M=\{i_1,\dots, i_k\}$.

For $b=(b_1,\dots, b_k)\in \C^k$, there is an isomorphism $\theta_b$ of $U(\mg_n)_{S(I)}$ such that
$$\theta_b(u)=\sum_{j_1,\dots,j_k\geq 0}\binom{b_1}{j_1}\cdots\binom{b_k}{j_k}
\text{ad}X_{-2\epsilon_{i_1}}^{j_1} \cdots\text{ad}X_{-2\epsilon_{i_k}}^{j_k} (u) X_{-2\epsilon_{i_1}}^{-j_1}\cdots X_{-2\epsilon_{i_k}}^{-j_k},$$
for any $u\in U(\mg_n)_S $. In particular, we can check that
\begin{equation}\label{Iso1}  \theta_b(X_{\epsilon_{-i_l}})=X_{-\epsilon_{i_l}}, \ \
 \theta_b(X_{\epsilon_{i_l}})=X_{\epsilon_{i_l}}+b_lX_{-\epsilon_{i_l}}X_{-2\epsilon_{i_l}}^{-1},\end{equation}
\begin{equation}\label{Iso} \theta_b(X_{2\epsilon_{i_l}})=X_{2\epsilon_{i_l}}+2b_l(b_l-1-2h_{i_l})X_{-2\epsilon_{i_l}}^{-1},\end{equation}
for any $1\leq l\leq k$.

For a $U(\mg_n)_{S(I)} $-module $M$, it can be twisted by $\theta_b$  to be a new
 $U(\mg_n)_{S(I)} $-module $M^{\theta_b}$.  As vector spaces $M^{\theta_b}= M$. For $v\in M$, $x\in\mg_n$,
 $x\cdot v=\theta_b(x)v $. The modules $M$ and  $M^{\theta_b}$ are said to be equivalent.

\begin{lemma} Let $M$ be a simple Harish-Chandra
$\mg_n$-module. If both $X_\alpha$ and  $ X_{-\alpha}$ act on $M$ injectively, then both act on $M$ bijectively.
\end{lemma}

\begin{proof} Consider the two injective linear maps  $X_\alpha: M_\lambda\to M_{\lambda+\alpha}$ and $X_{-\alpha}: M_{\lambda+\alpha}\to M_\lambda$. Since $\dim M_\lambda<\infty$ and $\dim M_{\lambda+\alpha}<\infty$, we see that $\dim M_\lambda=\dim M_{\lambda+\alpha}$, and both $X_\alpha$ and  $ X_{-\alpha}$ act on $M$ bijectively.
\end{proof}

\begin{theorem}\label{p14}
Let $M$ be a simple Harish-Chandra
$\mg_n$-module, on which  $z$ acts as   zero. Then $H_nM=0$, i.e., $M$ is a simple Harish-Chandra $\sp_{2n}$-module.
\end{theorem}

\begin{proof}
We can now consider $M$   as a module over $\sp_{2n}\ltimes \C^{2n}$.
By the structure of the Weyl group of $\sp_{2n}$, there is an inner automorphism $\sigma$ of $\sp_{2n}$ such that
$$\sigma(\mg_{2\epsilon_i})=\mg_{2\epsilon_i}, i\in I_M \cup F_M \cup F^+_M,  $$
and $$\sigma(\mg_{2\epsilon_i})=\mg_{-2\epsilon_i}, \sigma(\mg_{-2\epsilon_i})=\mg_{2\epsilon_i},\ i\in  F^-_M.$$
 Since $\sigma$ is an inner automorphism,
 we can extend $\sigma$ to be an automorphism $\tilde{\sigma}$ of $\sp_{2n}\ltimes \C^{2n}$ such that
 $\tilde{\sigma} (\C^{2n})= \C^{2n}$. The module $M$ can be twisted by $\tilde{\sigma}$  to be a new
 $\mg_n$-module $M^{\tilde{\sigma}}$.  Then $H_n M=0$ if and only if $H_n\cdot M^{\tilde{\sigma}} =0$.
 For $M^{\tilde{\sigma}}$, $$ I_{M^{\tilde{\sigma}}}\cup F_{M^{\tilde{\sigma}}} \cup F^+_{M^{\tilde{\sigma}}}=\{1,\dots,n\} .$$

  For convenience,  we can assume that
$$ I_M \cup F_M \cup F^+_M=\{1,\dots,n\} .$$

Let $L$ be the subalgebra of $\mg_n$ generated by
 $$ X_{\epsilon_i-\epsilon_j}, X_{\epsilon_i+\epsilon_j}, X_{-\epsilon_i-\epsilon_j},X_{-\epsilon_i}, X_{\epsilon_i}, \ i,j \in I_M,$$
which is isomorphic to $\mg_k$.

If $F_M \cup F^+_M=\emptyset$, let $N=M$. Otherwise
set $$N= \{ v\in M \mid X_{\epsilon_i}v= X_{2\epsilon_i} v=0 \ \text{for any}\   i\in F_M \cup F^+_M\}.$$
Since $ X_{\epsilon_i}, X_{2\epsilon_i}$ act
locally nilpotent on $M$ and $[X_{\epsilon_i}, X_{2\epsilon_i}]=0$ for  any $i\in F_M \cup F^+_M$,
the space $N$ is nonzero.
Note that $N$ is a uniformly bounded $\sp_{2n}$-module. Then
   $N$ is a uniformly bounded $L$-module.
Choose a weight vector $v$ which is a common eigenvalue of $X_{2\epsilon_i}X_{-2\epsilon_i}$ for any $i\in I_M$.
Assume that
$X_{2\epsilon_i}X_{-2\epsilon_i}v=a_i v, a_i\in \C$. Let $w=X_{-2\epsilon_{i_1}}\cdots X_{-2\epsilon_{i_k}}v$. Take $b_i$ so that $2b_l(b_l-1-2h_{i_l})w=a_{i_l}w$.
Then by the formula (\ref{Iso}),
we see that $\theta_b ( X_{2\epsilon_i}) w=0$ for any  $i\in I_M$.
 Now consider the $L$-module $N^{\theta_b}$. We can check that
 $$N'=\{ v\in N^{\theta_b}\mid  X_{2\epsilon_i}\ \text{acts locally nilpotent on }\  N^{\theta_b}, \forall\ i\in I_M \}$$
 is a nonzero $L$-submodule of $N^{\theta_b}$. Being uniformly bounded, $N'$ has a simple $L$-submodule $N''$.
 By (a) in Lemma \ref{lemma-nil}, $X_{\epsilon_i} $ acts locally nilpotent on  $N''$ for any  $i\in I_M$.
  Then by Proposition \ref{CGHM}, $N''$ is a simple generalized highest
 weight $L$-module.
 By Proposition \ref{GHM} (a), we deduce that $$\theta_b(X_{\epsilon_i}) N''=\theta_b(X_{-\epsilon_i}) N''=0 ,\forall i\in I_M.$$  By the formula
 (\ref{Iso1}), we see that
 $$X_{\epsilon_i}N''=X_{-\epsilon_i}N''=0 ,\forall i\in I_M.$$
 Consequently $X_{\epsilon_i} N''=0$, and hence $X_{\epsilon_i}$ acts locally nilpotent on $M$ for any  $i=1,\dots,n$.
 Similarly, we can prove $X_{-\epsilon_i}$ acts locally nilpotent on $M$ for any  $i=1,\dots,n$. Then
 $$W=\{v\in M\mid X_{-\epsilon_i}v= X_{\epsilon_i}v=0,\ \forall\ i=1,\dots,n\}$$ is nonzero, which is a $\mg_n$-submodule of
 $M$. The simplicity of $M$ tells us that $W=M$, i.e., $H_n M=0$.
\end{proof}

It is known (see \cite{FGM}) that any irreducible weight module over  the  Weyl algebra $\md_{n}$ is isomorphic to some irreducible subquotient
of the module $$F(a):=t^a\C[t_1^{\pm 1}, \dots, t_n^{\pm1}],$$  for some $a=(a_1,\dots,a_n)\in\C^n,$  where $ t^a=t_1^{a_1}\cdots t_n^{a_n}$.
Since $F(a)\cong F(a+\beta)$ for any $\beta\in \Z^n$, we always  assume that $a_i=0$ if $a_i\in\Z$.
 Set $$\text{Int}_a=\{i \mid a_i\in\Z\}.$$ Assume that $\text{Int}_a= \{j_1,\dots, j_l\}$.
Let \begin{equation}G(a)=F(a)/(\md_n\C[t_{j_1}]+\cdots+ \md_n\C[t_{j_l}]).\end{equation}
We can see that $G(a)$ is a simple $\md_n$-module.

For any simple Harish-Chandra
$\mg_n$-module $M$,  we have defined the set
 $$I_M= \{ 1\leq i \leq n\mid  X_{2\epsilon_i} |_M\   \text{and} \ X_{-2\epsilon_i}|_M \ \text{are injective}\}.$$

\begin{theorem}\label{t15}
Let $M$ be a simple Harish-Chandra
$\mg_n$-module with a nonzero central charge, say $\dot{z}\in \C^*$.
\begin{enumerate}[$($a$)$]
\item If  $I_M=\{1,\dots,n\}$, then
$M$ is equivalent to $N\otimes_{\z} F(a) $ for some finite dimensional  simple
$\sp_{2n}$-module $N$ and   $a=(a_1,\dots,a_n)\in\C^n$ with $a_1,\dots,a_n\not\in\Z$.
\item If  $I_M=\emptyset$, then there is  a simple Harish-Chandra  $\gl_n$-module $V$ such that $M$ is equivalent (under an inner automorphism of $\mg_n$) to
the generalized highest weight module $L(\z, V)$ defined in  subsection \ref{GHWD}.
\item  If  $I_M$ is a nonempty proper subset of $\{1,\dots,n\}$,  then there is a  simple Harish-Chandra  $\gl_n$-module $V$ and $a=(a_1,\dots,a_n)\in\C^n$ such that
$M$ is equivalent (under an inner automorphism of $\mg_n$) to $L_{\sp_{2n}}(V)\otimes_{\z} G(a) $, where $\text{Int}_a=\{1,\dots,n\}\setminus  I_M$.
\end{enumerate}
\end{theorem}

\begin{proof}  For each $i: 1\leq i\leq n$,  either $X_{\e_i}$ or $ X_{-\e_i}$  acts injectively on $M$. Otherwise $\z =0$ since there is no nontrivial finite-dimensional $\md_n$-module,   which contradicts our assumption.
So $F_M=\emptyset$.

(a) By Lemma \ref{lemma-nil}, both $X_{\e_i}$  and $X_{-\e_i}$  act injectively on $M$. Let
$v\in M$ be a nonzero weight vector which is a common eigenvector of $ X_{\e_i}X_{-\e_i}$ for all $i$. Then by the classification of simple weight modules over the Weyl algebra
 $\md_{n}$ (see \cite{FGM}),    there is an $a=(a_1,\dots,a_n)\in\C^n$  with $a_1,\dots,a_n\not\in\Z$
 such that $U(H_n)v \cong F(a)$. Note that
 $F(a)$ is a $U(H_n)$-module via the map $\theta_{\z}$. By the isomorphism $\phi_{\z}$
 in Proposition \ref{proposition3} and the dimension finiteness of weight spaces, we must have
 $M\cong  {N}\otimes_{\z} F(a) $ for some finite dimensional  simple
$\sp_{2n}$-module $N$.

(b)  Since  $I_M=\emptyset$, after twisting $M$ by an inner automorphism, we can assume that
$  F^+_M=\{1,\dots,n\}$, that is, $X_{2\epsilon_i}$ acts  locally nilpotently on $M$ for any $1\leq i\leq n$.
By Proposition \ref{CGHM}, there is a simple Harish-Chandra
$\gl_n$-module $V$ such that $M$ is equivalent to the generalized highest weight module $L(\z, V)$.

(c)   Suppose that $I_M=\{i_1,\dots, i_k\}$. After twisting $M$ by an inner automorphism (similar arguments to those in the proof of Proposition \ref{p14}), we can assume that
$$ I_M \cup F^+_M=\{1,\dots,n\}.$$
Here we have replaced our old $M$ with an equivalent module.

Choose a nonzero weight vector $ v\in M $ such that  $X_{\epsilon_i}v=0$ for any $ i \in F^+_M$.
We can further choose $v$ so that it is a common eigenvector of $ X_{\e_i}X_{-\e_i}$ for all $i$.
Then we can see that there is an $a=(a_1,\dots,a_n)\in\C^n$
 such that $U(H_n)v \cong G(a)$ with $a_i=0$ for $i\in F^+_M$.

Using the isomorphism in Proposition \ref{proposition3}, there is a simple $\sp_{2n}$-module $N$ such that $M$ is
 equivalent to $N\otimes_{\z} G(a)$.
 Note that $$\text{Supp} ( G(a))=\sum_{i\in I_{M}}(a_i\e_i+\Z \e_i)+\sum_{j\not\in I_{M}}(-\mathbb{N}\e_j).$$
 Since all the weight spaces of $N\otimes_{\z} G(a)$ are finite dimensional, there is a weight  $\lambda$ of $N$
 such that $\lambda+2\e_i\not\in \text{Supp}(N)$ for any $i\in\{1,\dots,n\}$. This implies that there are nonzero
 vectors in $N$ which are annihilated by $\mg_n^+\cap\sp_{2n}$.
 Therefore $N$ is isomorphic to the generalized highest weight module $L_{\sp_{2n}}(V)$. See Proposition \ref{CGHM}.
\end{proof}

So far we have obtained the classification for all simple Harish-Chandra modules over the symplectic oscillator  Lie  algebra $\mg_n$ for all $n$. Theorem \ref{p14} gave all such $\mg_n$-modules with $z$ acts trivially,  which are actually simple parabolically induced $\sp_{2n}$-modules and  simple cuspidal $\sp_{2n}$-modules (See \cite{M}).   Theorem \ref{t15} gave all such $\mg_n$-modules with $z$ acts non-trivially, which consists of three classes: cuspidal $\mg_n$-modules in (a), parabolically induced   $\mg_n$-modules in (b), and a third class described in (c). We can see that the  third class in Theorem \ref{t15} (c) does not appear for finite-dimensional simple Lie algebras.

\vspace{2mm}
\noindent
{\bf Acknowledgements. } This research is partially supported by NSFC (11771122, 11871190) and NSERC (311907-2015).

\end{document}